\documentclass[12pt]{amsart}
\usepackage{amsfonts,amssymb,amsmath,mathrsfs}
\usepackage{enumerate}
\usepackage{float}
\allowdisplaybreaks
\usepackage{longtable}
\usepackage{amsfonts,amssymb,amsmath,textcomp}
\usepackage{supertabular}
 2

\usepackage[autostyle]{csquotes}
\theoremstyle{plain}
\newtheorem{thm}{Theorem}[section]

\newtheorem{lem}[thm]{Lemma}
\newtheorem{prop}[thm]{Proposition}
\newtheorem{cor}[thm]{Corollary}

\theoremstyle{definition}

\newcommand{\Z}{\mathbb{Z}}
\newcommand{\Q}{\mathbb{Q}}

\begin{document}
\baselineskip 6.1mm

\title
{Primes of higher degree and Annihilators of Class groups}
\author{Nimish Kumar Mahapatra, Prem Prakash Pandey, Mahesh Kumar Ram}

\address[Nimish Kumar Mahapatra]{Indian Institute of Science Education and Research, Berhampur, India.}
\email{nimishkm18@iiserbpr.ac.in}

\address[Prem Prakash Pandey]{Indian Institute of Science Education and Research, Berhampur, India.}
\email{premp@iiserbpr.ac.in}

\address[Mahesh Kumar Ram]{Indian Institute of Science Education and Research, Berhampur, India.}
\email{maheshk@iiserbpr.ac.in}


\date{\today}


\begin{abstract}
Let $L/K$ be a Galois extension of number fields with Galois group $G$. We discuss a new method to obtain elements in $\Z[G]$ which annihilate the class group of $L$. Using this method, we obtain annihilators of class groups of cyclotomic fields. We show that these annihilators are new.  Some more consequences are also discussed. Moreover, we mention some results and connections to highlight importance of primes of higher residue degree.
\end{abstract}

\maketitle{}

\section{Introduction}
Let $L/K$ be a Galois extension of number fields with $Gal(L/K)=G$.  The group ring $\mathbb{Z}[G]$ acts on the set of non-zero  ideals of $L$ (integral or fractional). For any ideal $\mathfrak{a}$ of $L$ and an element $\theta =\sum_{\sigma \in G} n_{\sigma} \sigma \in \mathbb{Z}[G]$ the action is defined by 
$$\theta (\mathfrak{a})=\prod_{\sigma \in G} \sigma(\mathfrak{a})^{n_{\sigma}}.$$ 

Elements $\theta \in \mathbb{Z}[G]$ for which $\theta(\mathfrak{a})$ is a principal ideal for all non-zero ideals $\mathfrak{a}$ of $L$ are called relative annihilators of the class group of $L$ for the extension $L/K$.  If $K=\Q$, then we simply call them annihilators of the class group of $L$. Annihilator of class groups of cyclotomic fields are very important (see their use in resolution of Catalan's conjecture \cite{YFB04, BBM14, RS08}). There is a vast literature on the study of annihilators of class groups \cite{FT88,LW91, JT81, JSBT18, CGRC21}. When $K$ is of class number one then the $G-$trace $N=\sum_{\sigma \in G} \sigma$ in $\mathbb{Z}[G]$ is an annihilator of class group of $L$ for the the extension $L/K$. 

For the cyclotomic field $\Q(\zeta_n)$, the class number will be denoted by $h_n$. We will use $h_n^{+}$ to denote the class number of the maximal real subfield $\Q(\zeta_n^{+})$. For any prime ideal $\wp $ of $L$ and a prime $\mathfrak{p}$ of $K$ below $\wp$ we denote the degree $[\mathcal{O}_L/\wp: \mathcal{O}_K/\mathfrak{p}]$ by $Res_{L/K}\wp$. When there is no confusion, we should simply refer this by the residue degree of $\wp$.  All the prime ideals appearing in this article are unramified. 

For any odd prime number $\ell$,  Kummer proved that the class group of $\mathbb{Q}(\zeta_{\ell})$ is generated by prime ideals of residue degree one (see Chapter 8 of \cite{RS08}).  As a consequence of class field theory, much stronger result is known.  We recall the following result [Theorem 4.6 , \cite{GJ96}].
\begin{thm}\label{CFT1}
Every ideal class in the class group of $L$ contains infinitely many prime ideals $\mathfrak{p}$ of residue degree one
\end{thm}

Prime ideal of higher residue degree are not that well explored. In this article, we aim to highlight that their study leads to important results. In this direction, the second author proposed a new idea to construct annihilators of class groups for certain relative extension of number fields \cite{PPP19}. However not a single example was known, until now, where these annihilators were shown to exist. The aim of this article is to provide several class of examples where the idea of \cite{PPP19} succeeds in giving annihilators. Moreover, we also extend these constructions for a wider class of extensions. Also, we shall demonstrate that the annihilators obtained from this method are new. 

Let $C\ell(L)$ denote the class group of $L$. In \cite{PPP19} the second author initiated the study about the set 
\begin{equation*}
\mathbf{R}_{L/K}=\{f \in \mathbb{N}: C\ell(L) \mbox{ is generated by unramified primes of residue degree }f\}.
\end{equation*}

From Theorem \ref{CFT1} it follows that $1 \in \mathbf {R}_{L/K}$.  The case of $\mathbf{R}_{L/K}=\{1\}$ will be referred as trivial case.  In this article, for the first time, we give examples of extensions $L/K$ for which $\mathbf{R}_{L/K}$ is non-trivial. We provide several family of examples. Moreover, we obtain some results to show that the sets $\mathbf{R}_{L/K}$ themselves are important object of study. \\

For any $f \in \mathbf{R}_{L/K}$ we denote all the cyclic subgroups of $G$ of order $f$ by $H_f^1,  H_f^2, \ldots, H_f^r$. Put $H=\bigcap_{i=1}^r H_f^i$. For each $i$, let $\mathscr{S}_i$ denote the collection of all sets $S_i$ such that \\
(i) the set $S_i$ is a complete set of coset representatives of cosets of $H$ in $G$,\\
(ii) the set $S_i$ is product of a complete set of coset representatives of $H_f^i$ in $G$ and a complete set of coset representatives of $H$ in $H_f^i$. \\
Then we put $$\mathscr{S}=\bigcap_{i=1}^r \mathscr{S}_i.$$
If $G$ is cyclic then the set $\mathscr{S}$ is non-empty for each $f$. There are several examples of non-cyclic $G$ and $f$ for which the set $\mathscr{S}$ is non-empty. For example, $G=\{\pm 1, \pm i, \pm j, \pm k\}$ with the multiplication rules $i^2=j^2=k^2=-1, i.j=k, j.k=i, k.i=j, j.i=-k,k.j=-i,i.k=-j$ and $f=4$.\\

For any $S \in \mathscr{S}$,  the set S is a complete set of coset representatives of cosets of $H$ in $G$. Also, for each $i$,  there exist a complete set of coset representatives of $H_f^i$ in $G$, say, $T_i$ and a complete set of coset representatives of $H$ in $H_f^i$, say, $T_i'$ such that
\begin{equation}\label{ann1}
S =T_i T_i'
\end{equation}
For each $S \in \mathscr{S}$, define 
$$\theta_f(S)=\sum_{\sigma_i \in S} \sigma_i.$$

\begin{thm}\label{MT2}
If the class number of $K$ is one then $\theta_f(S)$ annihilates the class group of $L$.
\end{thm}

Theorem \ref{MT2} generalises Theorem 1.2 of \cite{PPP19}.  If the extension $L/K$ is cyclic then the annihilators $\theta_f(S)$ are same as the annihilators $\theta_f$ introduced in \cite{PPP19}.  For cyclic extensions, we shall use the notation $\theta_f$ over $\theta_f(S)$. Also, for $f=1$ we use the notation $\theta_1$ instead of $\theta_1(S)$. The main advantage of Theorem \ref{MT2} is that it can be used to get annihilators for relative extensions. In fact, some of the annihilators $\theta_f(S)$ are in the group ring $\Z[G']$ where $G'=Gal(L/F)$ for some intermediate field $F$.  Annihilators $\theta_f(S)$ are almost like the $G-$trace $N=\sum_{\sigma \in G} \sigma$.  However they do not have the obvious drawback of $G-trace$, namely, $\sigma N=N$ for any $\sigma \in G$.  Moreover, these annihilators are new: in Section 4, we show that the annihilators $\theta_f(S)$ may not appear in the Stickelberger ideal for the cyclotomic fields. \\


Our first result, with non-trivial $\mathbf{R}_{L/K}$, is the following theorem.

\begin{thm}\label{MT1}
Let $q \equiv 7 \pmod 8$ be an odd prime different from $7$. Assume that the class number of $\mathbb{Q}(\zeta_q)$ is prime. If $(q-1)/2$ is relatively prime to the class number of $\mathbb{Q}(\sqrt{-q})$ then $\mathbf{R}_{\mathbb{Q}(\zeta_q)/\mathbb{Q}}$ is non-trivial. In fact,  $\mathbf{R}_{\mathbb{Q}(\zeta_q)/\mathbb{Q}}$ contains a divisor of $(q-1)/2$ other than one.
\end{thm}

It is not easy to find primes $q$ satisfying hypothesis of Theorem \ref{MT1} because of primality condition on $h_q$. In fact, $q=23$ is the only known example. It may be possible that there are no other examples.  To remedy this, we look towards real cyclotomic fields. The class numbers $h_n^+$ are primes more often than the class numbers $h_n$.  In the next theorem we describe some real cyclotomic fields such that $\mathbf{R}_{\Q(\zeta_n^{+})/\Q}$ is non-trivial. 

\begin{thm}\label{MT3}
Let $q$ be an odd prime and $m>1$ be an integer such that $n=(2mq)^2+1$ is square-free. Assume that the class number $h_n^+$ is prime. If $\phi(n)/4$ and the class number of $\Q(\sqrt{n})$ are relatively prime then $\mathbf{R}_{\Q(\zeta_n^{+})/\Q}$ is non-trivial. 
\end{thm}

Analogously, we also have the following theorem.
\begin{thm}\label{MT31}
Let $n=[(2m+1)q]^2+4$ be a square-free integer, where $q$ is an odd prime and $m \geq 1$. If $h_{n}^{+}$ is prime and $\phi(n)/4$ is relatively prime to the class number of $\Q(\sqrt{n})$ then $\mathbf{R}_{\Q(\zeta_{n}^{+})/\Q}$ is non-trivial. 
\end{thm}

The study of sets $\mathbf{R}_{L/K}$ is very important. Apart from providing annihilators, they are nicely interlinked with class numbers of subfields. Prime ideals of higher residue degree are not well explored and we hope these results will bring the set $\mathbf{R}_{L/K}$ at forefront.  We mention some such results.

\begin{thm}\label{MT4}
Let $n$ be a positive integer such that $(\Z/n\Z)^{*}$ is cyclic. Let $f$ be a divisor of $\phi(n)$ which is relatively prime to $\phi(n)/f$.  We use $K_f$ to denote the subfield of $\Q(\zeta_n)$ such that $[K_f:\Q]=f$. If $f$ is in $\mathbf{R}_{\Q(\zeta_n)/\Q}$ then the $q-$part of the class group of $K_f$ is trivial for all $q$ not dividing $\phi(n)/f$.
\end{thm}

Analogous result holds for the extension $\Q(\zeta_n^+)/\Q$. We mention some immediate corollaries of Theorem \ref{MT4}.

\begin{cor}\label{MT11}
Let $n=p^r$ be an odd prime power, where $p \equiv 3 \pmod 4$. If $2 \in \mathbf{R}_{\mathbb{Q}(\zeta_n)/\mathbb{Q}}$ then the class number of the quadratic field contained in $\Q(\zeta_n)$ is trivial outside $\phi(n)/2$.
\end{cor}

\begin{cor}\label{MT12}
Let $\ell\equiv 3 \pmod 4$ be a prime such that $(\ell-1)/2 \in \mathbf{R}_{\Q(\zeta_{\ell})/\Q}$ then $h_{\ell}^{+}$ is a power of $2$.
\end{cor}

\begin{cor}\label{MT13}
Let $\ell\equiv 5 \pmod 8$ be a prime such that $2 \in \mathbf{R}_{\Q(\zeta_{\ell}^{+})/\Q}$ then the class number of $\Q(\sqrt{\ell})$ is trivial outside $(\ell -1)/4$. In particular it is odd.
\end{cor}

Section 2 records some preliminaries needed. All the proofs are given in Section 3. Section 4 discusses some cyclotomic fields for which our methods do produce annihilators. We show, by an example, that these annihilators are not in the Stickelberger ideal.  Apart from the examples given here, some more examples of extension $L/K$ with non-trivial $\mathbf{R}_{L/K}$ will appear in third author's thesis. 


\section{Preliminaries}
We begin with stating the quadratic reciprocity law.  For any pair of odd primes $p, q$ we use $(\frac{p}{q})$ to denote the Legendre symbol. The law of quadratic reciprocity states the following
$$\left(\frac{p}{q}\right)\left(\frac{q}{p}\right)=(-1)^{\frac{p-1}{2}\cdot\frac{q-1}{2}}.$$
If both $p$ and $q$ are congruent to $3$ modulo $4$ then the congruence $x^2 \equiv p \pmod q$ is solvable if and only if the congruence $x^2 \equiv -q \pmod p$ is solvable. When at least one of $p$ and $q$ is congruent to $1$ modulo $4$ then the congruence $x^2 \equiv p \pmod {q}$ is solvable if and only if the congruence $x^2 \equiv q \pmod p$ is solvable.  


In \cite{AG06}, Gica proved the following result.
\begin{prop}\label{P3}
Let $q $ be a prime different from $2,3,5,7, 17$. There exists a quadratic residue $p$ modulo $q$ such that $p\equiv 3 \pmod 4$ and $p<q$. 
\end{prop}
If $q \equiv 7 \pmod 8$ is a prime different from $7$, then using Proposition \ref{P3} and law of quadratic reciprocity we see that there exists a prime $p<q$ such that $p \equiv 3 \pmod 4$ and $-q$ is a quadratic residue modulo $p$.


We record the following proposition from \cite{ACH65}.
\begin{prop}\label{P4}
For any pair of integers $m$ and $n$ such that $m$ is not a square and $m<2n$, the diophantine equation 
$$x^2-(n^2+1)y^2=\pm m$$ has no solution in integers. 
\end{prop}

Along the similar lines we have following result from \cite{SDL77}.
\begin{prop}\label{P5}
Let $m$ and $n$ be integers such that $m$ is not a square, $n \geq 2$ and $m<n$. The diophantine equation 
$$x^2-(n^2+4)y^2=\pm 4 m$$ has no solution in integers. 
\end{prop}

We mention the following elementary group theoretic lemma.

\begin{lem}\label{L1}
Let $G$ be an abelian group of order $n$. Let $H_1$ be a subgroup of order $r$ and $H_2$ be the subgroup of order $s=n/r$. If $r$ and $s$ are relatively prime then for any $\sigma \in H_1$ the set $\sigma H_2$ is a set of representative for $G/H_1$.
\end{lem}

%

\section{Proofs}
Before proving Theorem \ref{MT2} we remark that $\theta_1=N$, the $G-$trace. There may be $f \in \mathbf{R}_{L/K}$ such that $f>1$ and $\theta_f(S)=N$. However, $\theta_f(S) \not = N$ for cyclotomic fields of prime power conductor whenever $f>1$.  Consider a cyclotomic field $\Q(\zeta_{p^r})$ of prime conductor. Then the Galois group $Gal(\Q(\zeta_{p^r})/\Q)$ has a unique subgroup of order $f$, say $H$. Any set of representatives of the corresponding quotient group will be a proper subset of $Gal(\Q(\zeta_{p^r})/\Q)$.

\begin{proof}(Proof of Theorem \ref{MT2})
Let $\wp$ be an unramified prime ideal of $L$ of residue degree $f$. We let $D_f$ denote the decomposition subgroup of $G$ for the prime ideal $\wp$.  Let $H_f^1, \ldots, H_f^r$ denote all the cyclic subgroups $G$ of order $f$.  As $D_f$ is a cyclic subgroup of order $f$, we must have $D_f=H_f^i$ for some $i$.  From (\ref{ann1}), it follows that there exists a complete set of coset representatives of $D_f$ in $G$, say, $T_i$ and a complete set of coset representatives of $H$ in $D_f$, say, $T_i'$ such that
\begin{equation}\label{emt200}
\theta_f(S)=\sum_{\sigma \in T_i, \sigma' \in T_i'} \sigma \sigma'.
\end{equation}

We write $$N_f^i=\sum_{\sigma \in T_i}\sigma.$$ 
Then
\begin{equation}\label{emt20}
\theta_f(S)=\sum_{\sigma' \in T_i'} N_f^i\sigma'.
\end{equation}
Thus 
\begin{equation}\label{emt21}
\wp^{\theta_f(S)}=\prod_{\sigma' \in T_i'} \sigma'(\wp^{N_f^i}).
\end{equation}
If $\mathfrak{p}$ is the prime ideal of $K$ lying below $\wp$ then, from the factorization theorem of Dedekind, we see that
$$\mathfrak{p} \mathcal{O}_L=\wp^{N_f^i}.$$
Combining this with equation (\ref{emt21}) we obtain
$$\prod_{\sigma' \in T_i'} \sigma'(\mathfrak{p} \mathcal{O}_L)=\wp^{\theta_f(S)}.$$
As the class number of $K$ is one, the ideal $\mathfrak{p} \mathcal{O}_K$ is principal, and so is the ideal $\mathfrak{p} \mathcal{O}_L$. As a result we see that the ideal $\wp^{\theta_f(S)}$ is principal. Now the theorem follows from our assumption that the class group is generated by prime ideals of residue degree $f$.
\end{proof}

Now we proceed with the proof of Theorem \ref{MT1}. We use the arithmetic of the quadratic subfield and Proposition \ref{P3}.

\begin{proof}
(Proof of Theorem \ref{MT1}) \\
Let $p$ be a prime as in Proposition \ref{P3}.  Then there exists an integer $a$ such that 
$$a^2 \equiv p \pmod {q}.$$
From this it follows that
\begin{equation}\label{EMT1}
p^{\phi(q)/2}=1 \pmod {q}.
\end{equation}
As $p < q$, we conclude that the residue degree of $p$ in $\Q(\zeta_{q})$ is strictly more than $1$ and strictly less than $\phi(q)$.  Let $f$ be the residue degree of $p$ in $\Q(\zeta_{q})$, then $f>1$ is a divisor of $\phi(q)/2$.

As remarked after Proposition \ref{P3}, $-q$ is a quadratic residue modulo $p$. It follows that $p$ factors as a product of two distinct prime ideals in $\Q(\sqrt{-q})$, say $p\Z[(1+\sqrt{-q})/2]=\mathfrak{p}_1\mathfrak{p}_2$. 

If ideal $\mathfrak{p}_1$ is principal, then taking norm we obtain integers $x, y$ such that
$$x^2+q y^2=4p.$$
As $p< q$, we must have $y=1$ and this gives
$$x^2+q =4p.$$
Reading the above modulo $8$ leads to a contradiction and thus $\mathfrak{p}_1$ is a non-principal ideal.

Let $\wp$ be a prime ideal of $\Q(\zeta_q)$ above $p$. Then the residue degree of $\wp$ is $f$. We claim that $\wp$ is not principal. As the class number of $\Q(\zeta_{q})$ is assumed to be prime, it follows that $\wp$ generates the class group and hence $f \in \mathbf{R}_{\mathbb{Q}(\zeta_{q})/\mathbb{Q}}$.

If $\wp$ is principal then, by taking norm, we see that $\mathfrak{p}_1^f$ is principal. As $\mathfrak{p}_1$ is non-principal, it follows that $f$ divides the class number of $\Q(\sqrt{-q})$. But this is not possible as $f\mid\phi(q)/2$ and we have assumed that $\phi(q)/2$ is relatively prime to the class number of $\Q(\sqrt{-q})$.\\

As $q\equiv 7 \pmod 8$, the following factorisation holds:
$$2\Z[(1+\sqrt{-q})/2]=\mathfrak{p}_1'\mathfrak{p}_2'.$$
The Pell equation $x^2+qy^2=2 \mbox{ or }8$ is not solvable for $q\not = 7$ and hence the ideal $\mathfrak{p}_1'$ is non-principal. Let $f'$ be the residue degree of $2$ for the extension $\Q(\zeta_q)/\Q$. Then, as argued for $f$, it follows that $f' \in \mathbf{R}_{\mathbb{Q}(\zeta_q)/\mathbb{Q}}$. Thus $1,f,f' \in \mathbf{R}_{\mathbb{Q}(\zeta_q)/\mathbb{Q}}$.


%
\end{proof}

From the proof of Theorem \ref{MT1} it seems that if $q$ satisfies all the conditions of Theorem \ref{MT1} then $\mathbf{R}_{\Q(\zeta_q)/\Q}$ has $3$ or more elements. However, in the only known example we have $f=f'$. Proof of Theorem \ref{MT3} and Theorem \ref{MT31} are along the similar lines. We now give a proof of Theorem \ref{MT31} and omit the proof of Theorem \ref{MT3}.

\begin{proof}(Proof of Theorem \ref{MT31})
We first see that the prime $q$ splits in $\Q(\sqrt{n})$, say $q\Z[(1+\sqrt{n})/2]=\mathfrak{q}_1\mathfrak{q}_2$. From Proposition \ref{P5}, it follows that the diophantine equation
$$x^2-n y^2=\pm 4q$$ is not soluble. Consequently the ideal $\mathfrak{q}_1$ is non-principal. Now, as argued in the proof of Theorem \ref{MT1}, it can be shown that the prime ideal of $\Q(\zeta_{n}^{+})$ above $\mathfrak{q}_1 $ is non-trivial. As $h_{n}^{+}$ is prime, it follows that $f=Res_{\Q(\zeta_{n}^{+})/\Q} ~q$ is in the set $\mathbf{R}_{\Q(\zeta_{n}^{+})/\Q}$.

If $p$ is a prime divisor of $2m+1$, then, along the above arguments,  it follows that $f'=Res_{\Q(\zeta_{n}^{+})/\Q} ~p$ is in the set $\mathbf{R}_{\Q(\zeta_{n}^{+})/\Q}$.
\end{proof}

Let $L/K$ be abelian and $f \in \mathbf{R}_{L/K}$. Suppose $K_f$ denotes the subfield of degree $f$ over $K$ and $H=Gal(L/K_f)$. Then it is not always true that 
$$\theta_f(S)=\sigma \sum_{\tau \in H} \tau \mbox{ for some }\sigma \in G.$$ 
However if we assume that $G$ is cyclic and $f$ is relatively prime to $|G|/f$, then, from Lemma \ref{L1}, it follows that the elements $\sigma \sum_{\tau \in H} \tau \in \Z[H]$ arise as $\theta_f(S)$.

\begin{proof}(Proof of Theorem \ref{MT4})
Let $f \in \mathbf{R}_{\Q(\zeta_n)/\Q}$ and $K_f$ denote the subfield of $\Q(\zeta_n)$ of degree $f$. Let $H$ be the Galois group $Gal(\Q(\zeta_n)/K_f)$. As the Galois group $Gal(\Q(\zeta_n)/\Q)$ is cyclic, it has unique subgroup of order $f$. Now, by the Lemma \ref{L1} and Theorem \ref{MT2}, $\theta_f(S)= \sum_{\tau \in H} \tau$ is an annihilator of the class group of $\Q(\zeta_n)$.\\

Let $\mathfrak{p}$ be a prime ideal of $K_f$ and $\wp$ denote a prime ideal of $\Q(\zeta_n)$ lying above $\mathfrak{p}$. Note that 
$$Res_{\Q(\zeta_n)/K_f} \wp=t \Longrightarrow t\mid \frac{\phi(n)}{f} \mbox{ and } \wp^{\theta_f(S)}=\mathfrak{p}^t \Z[\zeta_n].$$
But $\wp^{\theta_f(S)} $ is principal, so $\mathfrak{p}^t \Z[\zeta_n]$ is principal. In particular, $$\mathfrak{p} ^{\phi(n)/f} \Z[\zeta_n]$$ is principal. Taking the norm, we obtain
$$N_{\Q(\zeta_n)/K_f}(\mathfrak{p}^{\phi(n)/f} \Z[\zeta_n])=\mathfrak{p}^{(\phi(n)/f)^2}$$
is principal. 

Thus, for any ideal $\mathfrak{p}$ of $K_f$ the ideal $\mathfrak{p}^{(\phi(n)/f)^2}$ is principal. So the order of the ideal class $[\mathfrak{p}]$ must be a divisor of $(\phi(n)/f)^2$. Now the theorem follows.\\

\end{proof}

Along these lines we can obtain similar result for the extension $\Q(\zeta_n)^{+}/\Q$. We state the result without the proof. Let $f$ be a divisor of $\phi(n)/2$ which is relatively prime to $\phi(n)/(2f)$. Let $K_f$ denote the subfield of degree $f$. If $f \in \mathbf{R}_{\Q(\zeta_n^+)/\Q}$ then the $q-$part of the class group of $K_f$ is trivial for all $q$ not dividing $\phi(n)/(2f)$.\\

\section{Cyclotomic fields}
In this section we analyse our results for specific cyclotomic fields. For all the class numbers of cyclotomic fields used here we refer to \cite{LW91} and \cite{JCM15}. In particular, we are using the following results:
\begin{enumerate}[(i)]
    \item If $n$ is a prime power with $\phi(n) \leq 66$ then $h_n^+=1$.
    \item If $n$ is not a prime power and $n \leq 200$ with $\phi(n) \leq 72$ then $h_n^+=1$ except for $n=136,148, 152$.
    \item For any prime $p$, we are using the tables from \cite{LW91} for $h_p^{+}$. We remark that the value of $h_p^{+}$ in these tables is still a conjecture \cite{RS03}.
\end{enumerate}
\subsection{The set $\mathbf{R}_{\Q(\zeta_n)/\Q}$}

The first cyclotomic field with non-trivial class group is $\Q(\zeta_{23})$.  The class number of this is $3$ and the calss number of the quadratic field $\Q(\sqrt{-23})$ is $3$. Thus Theorem \ref{MT1} applies to give us $11 \in \mathbf{R}_{\Q(\zeta_{23})/\Q}$. Also from Theorem \ref{MT4} it follows that $2 \not \in \mathbf{R}_{\Q(\zeta_{23})/\Q}$. Thus $\mathbf{R}_{\Q(\zeta_{23})/\Q}=\{1,11\}$.\\


Now we look at some real cyclotomic fields. Our first example is $\Q(\zeta_{257}^{+})$.  We have $h_{257}^{+}=3$ and the class number of $\Q(\sqrt{257})$ is $3$. We note that $2$ splits in $\Q(\sqrt{257})$ and, as $257=(16)^2+1$, from Proposition \ref{P4} we see that the prime above $2$ is not principal. The residue degree of $2$ in the extension $\Q(\zeta_{257}^{+})/ \Q $ is $8$. Along the lines of Theorem \ref{MT3}, we find $8 \in \mathbf{R}_{\Q(\zeta_{257}^{+})/ \Q}$.  Similarly, we find that $72 \in \mathbf{R}_{\Q(\zeta_{577}^{+})/ \Q}$ and $200 \in \mathbf{R}_{\Q(\zeta_{1601}^{+})/ \Q}$. Taking $m=9, q=5$ in Theorem \ref{MT3} we find $2025 \in \mathbf{R}_{\Q(\zeta_{8101}^{+})/\Q}$. \\

Now we look at examples of real cyclotomic fields coming from Theorem \ref{MT31}. The choice of $m=2,q=7$ gives $\ell =1229$. We have $h_{1229}^{+}=3$ and the class number of $\Q(\sqrt{1229})$ is $3$. All the conditions of Theorem \ref{MT31} are satisfied.  The residue degree of $7$ in the extension $\Q(\zeta_{1229}^{+})/\Q$ is $307$. Thus $307 \in \mathbf{R}_{\Q(\zeta_{1229}^{+})/\Q}$.  

Taking $m=4,q=5$ gives $169, 507 \in \mathbf{R}_{\Q(\zeta_{2029}^{+})/\Q}$.  Here $169$ is the residue degree of $3$ and $507$ is the residue degree of $5$.

The choice of $m=1, q=19$ gives $271, 813 \in \mathbf{R}_{\Q(\zeta_{3253}^{+})/\Q}$.  Here $271$ is the residue degree of $3$ and $813$ is the residue degree of $19$.

The choice of $m=1, q=11$ gives $7,  13 \in \mathbf{R}_{\Q(\zeta_{1093}^{+})/\Q}$.  Here $7$ is the residue degree of $3$ and $13$ is the residue degree of $11$.

The choice of $m=2, q=17$ gives $1807, 139  \in \mathbf{R}_{\Q(\zeta_{7229}^{+})/\Q}$.  Here $1807$ is the residue degree of $5$ and $139$ is the residue degree of $17$.

The choice of $m=2, q=19$ gives $61, 2257  \in \mathbf{R}_{\Q(\zeta_{9029}^{+})/\Q}$.  Here $61$ is the residue degree of $5$ and $2257$ is the residue degree of $19$.



\subsection{Annihilators $\theta_f(S)$}

In this subsection we give an example of annihilators $\theta_f(S)$ for cyclotomic field and show that it is not contained in the Stickelberger ideal of the cyclotomic field.  The example we consider are the annihilators $\theta_{11}$ for the extension $\Q(\zeta_{23})/\Q$.  Let $S$ denote the Stickelberger ideal of $\Q(\zeta_{23})$.  We write $Gal(\Q(\zeta_{23})/\Q)=\{\sigma_i: 1\leq i \leq 22\}$ where $\sigma_i(\zeta_{23})=\zeta_{23}^i$. We consider the following subsets of $\{1,2, \ldots, 22\}$.\\
$I_1=\{2,16,5,20,13,19,9,17,15,11,22  \}$, $I_2=\{3,18,7,21,13,19,9,17,15,11,22  \}$, $I_3=\{4,10,2,16,5,20,9,17,15,11,22  \}$, $I_4=\{2,16,5,20,13,19,9,17,15,11,22  \}$, $I_5=\{6,3,18,2,16,13,19,9,15,11,22  \}$, $I_6=\{10,7,21,5,20,19,9,17,15,11,22  \}$, $I_7=\{8,4,18,2,16,20,13,9,17,11,22  \}$, $I_8=\{3,21,16,5,13,19,9,17,15,11,22  \}$, $I_9=\{14,10,7,2,5,20,19,17,15,11,22  \}$, $I_{10}=\{18,21,16,20,13,19,9,17,15,11,22  \}$, $I_{11}=\{12,6,4,3,7,2,5,13,9,15,22  \}$.  \\
For each $j \in \{1, \ldots, 11\}$ we define $f_j=\sum_{i\in I_j}\sigma_i$. From Theorem 9.3 of \cite{RS08} we see that the elements $f_1, \ldots, f_{11}$ together with the $G-$ trace $ N=\sum_i\sigma_i$ form a $\Z-$ basis of the Stickelberger ideal $S$.  \\
If $\theta_{11} \in S$, then there exists integer $a_1, \ldots , a_{12}$ such that
\begin{equation}\label{e41}
\theta_{11}=a_1f_1+ \ldots a_{11}f_{11}+a_{12}N.
\end{equation}
If $ \mathfrak{p}$ is a prime ideal of $\Q(\zeta_{23})$ of residue degree $11$, then the decomposition group of $\mathfrak{p}$ is $$D_{\mathfrak{p}}=\{\sigma_1,\sigma_2,\sigma_4,\sigma_8,\sigma_{16},\sigma_9,\sigma_{18},\sigma_{13},\sigma_3,\sigma_6,\sigma_{12}\}.$$
A complete set of coset representatives of $G/D_{\mathfrak{p}}$ is $\{\sigma_1, \sigma_5\}$. By Lemma \ref{L1}, any other set of coset representative is a multiple of this set by an element of $D_{\mathfrak{p}}$. Thus we can assume that 
\begin{equation}\label{e42}
\theta_{11}=\sigma_1+\sigma_5.
\end{equation}
Comparing the coefficients of $\sigma_1$ in equations (\ref{e41}) and (\ref{e42}) gives us $a_{12}=1$. Comparing the coefficients of $\sigma_{12}$ in these two equation gives $a_{11}=-1$. Next we compare the coefficients of $\sigma_{11}$ to obtain
\begin{equation}\label{e43}
 1+a_1+\ldots+a_{10}=0.
\end{equation}
A comparison of coefficients of $\sigma_{22} $ in those two expressions leads to
\begin{equation}\label{e44}
1+a_1+\ldots+a_{11}=0.
\end{equation}
From equations (\ref{e43}) and (\ref{e44}) we obtain $a_{11}=0$ which contradicts to our earlier finding $a_{11}=-1$. Thus $\theta_{11} \not \in S$.

\section{Examples with non-trivial $\mathbf{R}_{F/\Q}$ and concluding remarks}
In this section we provide several examples of biquadratic extensions $F/\Q$ for which $\mathbf{R}_{F/\Q}$ is non-trivial. Let $F$ be a number field of degree $n$ over $\Q$ and discriminant $d_F$.  Suppose $n=r_1+2r_2$, where $r_1$ and $2r_2$ denote the number of real and complex embeddings of $F$ respectively. Then 
$$M_F=\frac{n!}{n^n}\left(\frac{4}{\pi}\right)^{r_2}\sqrt{\left|d_f\right|}$$
is called the Minkowski constant of $F$. It is well known that every ideal class of $F$ contains a prime ideal of norm not exceeding $M_F$.\\

\begin{prop}\label{P2}
The prime ideal of $\Q(\sqrt{79})$ above $3$ is non-principal.
\end{prop}
\begin{proof}
We note that the Minkowski constant for $F=\Q(\sqrt{79})$ is $9$. The factorisation of primes above $2,3,5,7$ is easily obtained:
$$2\mathcal{O}_F=\mathfrak{p}_2^2, 3\mathcal{O}_F=\mathfrak{p}_3\mathfrak{p}_3', 5\mathcal{O}_F=\mathfrak{p}_5 \mathfrak{p}_5', 7\mathcal{O}_F=\mathfrak{p}_7\mathfrak{p}_7'.$$
As the class number of $\Q(\sqrt{79})$ is $3$, it follows that $\mathfrak{p}_2$ is principal. Next we see that
\begin{equation}\label{ep21}
\langle 4+\sqrt{79}\rangle=\mathfrak{p}_3^2 \mathfrak{p}_7
\end{equation}
and
\begin{equation}\label{ep22}
\langle 3+\sqrt{79}\rangle=\mathfrak{p}_2 \mathfrak{p}_5\mathfrak{p}_7.
\end{equation}
If the ideal $\mathfrak{p}_3$ is principal then from (\ref{ep21}) we find $\mathfrak{p}_7 $ is principal. Now from (\ref{ep22}) we obtain that $\mathfrak{p}_5$ is principal. Consequently, all the primes of norm less than the Minkowski constant are principal. This contradiction establishes the proposition.
\end{proof}

\begin{thm}
Let $u$ be a square-free integer such that $u \equiv 2 \pmod 3$. Let $F$ denote the compositum of $K_u=\Q(\sqrt{u})$ and $\Q(\sqrt{79})$. If class number of $F$ is prime then both $\mathbf{R}_{F/\Q}$ and $\mathbf{R}_{F/\Q(\sqrt{79})}$ are non-trivial.
\end{thm}

\begin{proof}
Let $\wp$ denote a prime ideal of $F$ above $3$ and $\mathfrak{p}$ be the prime of $\Q(\sqrt{79})$ below $\wp$. We see that $3$ splits completely in $\Q(\sqrt{79})$ and remains inert in $K_u$. Consequently the residue degree of $\wp$ is $2$ for the extensions $F/\Q$ and $F/\Q(\sqrt{79})$. 

If $\wp$ is principal, then $$N_{F/\Q(\sqrt{79})} \wp= \mathfrak{p}^2$$ is also principal. As the class number of $\Q(\sqrt{79})$ is $3$, it follows that $\mathfrak{p}$ is principal. This contradicts the Proposition \ref{P2}. 

Thus $\wp$ is non-principal. Now the theorem follows from the assumption that the class number of $F$ is prime.
\end{proof}


We list these $u<1000$, for which the class number of $F$ is prime, in the following table.
\begin{center}
	\begin{table}[H]	
		\parbox[t]{0.69\linewidth}{
			\begin{tabular}[t]{|c|c|c|c||c|c|c|c|}
				\hline
				$u$ & $h_F$& $\mathbf{R}_{F/\Q}$ & $\mathbf{R}_{F/\Q(\sqrt{79})}$ & $u$ &$h_F$ &$\mathbf{R}_{F/\Q}$ & $\mathbf{R}_{F/\Q(\sqrt{79})}$ \\
				\hline
				2    &     3  & 1,2 &1,2 &                   5 & 3 &1,2 &1,2\\
				11    &    3  & 1,2 &1,2 &                  17 & 3 &1,2 &1,2\\
				23    &    3  & 1,2 &1,2 &                 29 & 3 &1,2 &1,2\\
				41    &    3  & 1,2 &1,2 &                 47 & 3 &1,2 &1,2\\
				53   &    3  & 1,2 &1,2 &                 59 & 3 &1,2 &1,2\\
				71   &    3  & 1,2 &1,2 &                 101 & 3 &1,2 &1,2\\
				107    &    3  & 1,2 &1,2 &                 113 & 3 &1,2 &1,2\\
				131    &    3  & 1,2 &1,2 &                 137 & 3 &1,2 &1,2\\
				149    &    3  & 1,2 &1,2 &                 158 & 3 &1,2 &1,2\\
				167    &    3  & 1,2 &1,2 &                 173 & 3 &1,2 &1,2\\
				197    &    3  & 1,2 &1,2 &                 227 & 3 &1,2 &1,2\\
				233    &    3  & 1,2 &1,2 &                 239 & 3 &1,2 &1,2\\
				251    &    3  & 1,2 &1,2 &                 263 & 3 &1,2 &1,2\\
				269    &    3  & 1,2 &1,2 &                 293 & 3 &1,2 &1,2\\	
				311   &    3  & 1,2 &1,2 &                 347 & 3 &1,2 &1,2\\
				383    &    3  & 1,2 &1,2 &                 395 & 3 &1,2 &1,2\\
				419    &    3  & 1,2 &1,2 &                 431 & 3 &1,2 &1,2\\
				449    &    3  & 1,2 &1,2 &                 461 & 3 &1,2 &1,2\\
				467    &    3  & 1,2 &1,2 &                 491 & 3 &1,2 &1,2\\
				503    &    3  & 1,2 &1,2 &                 509 & 3 &1,2 &1,2\\
				521    &    3  & 1,2 &1,2 &                 557 & 3 &1,2 &1,2\\	
				587    &    3  & 1,2 &1,2 &                 599 & 3 &1,2 &1,2\\
				647    &    3  & 1,2 &1,2 &                 677 & 3 &1,2 &1,2\\
				683    &    3  & 1,2 &1,2 &                 701 & 3 &1,2 &1,2\\
				719    &    3  & 1,2 &1,2 &                 743 & 3 &1,2 &1,2\\
				773 &    3  & 1,2 &1,2 &                 797 & 3 &1,2 &1,2\\
				821    &    3  & 1,2 &1,2 &                 827 & 3 &1,2 &1,2\\
				863    &    3  & 1,2 &1,2 &                 869 & 3 &1,2 &1,2\\
				887   &    3  & 1,2 &1,2 &                 911 & 3 &1,2 &1,2\\
				929    &    3  & 1,2 &1,2 &                 941& 3 &1,2 &1,2\\
				947    &    3  & 1,2 &1,2 &                 971 & 3 &1,2 &1,2\\
				983    &    3  & 1,2 &1,2 &                  &  & &\\
				
				\hline
			\end{tabular}
			
		}
	\end{table}
\end{center}

The annihilators constructed for these $f>1$ are all trivial, that is, the $G-trace$. But the main point here is to have non-trivial $\mathbf{R}_{F/\Q}$. In fact, if we can obtain infinitely many irreducible polynomials of the form $X^6-a$ such that, for $F=\Q(a^{1/6}, \zeta_6)$ and $K=\Q(\zeta_6)$, the set $\mathbf{R}_{F/K}$ is non-trivial. Then for the unique intermediate field $L$ of of degree $3$ of the extension $F/K$, the class number is trivial outside $2$.\\

We remark that all the families of extensions $L/K$ with non-trivial $\mathbf{R}_{L/K}$ obtained here is under the assumption that $h_L$ is prime. However, we can obtain annihilators similar to $\theta_f(S)$ for many more fields. We illustrate the simplest such construction. 

Suppose $L/K$ is an extension and $f>1$ is an element in $\mathbf{R}_{L/K}$. Let 
$$\theta_f(S)=\sigma_1+\ldots +\sigma_s$$ 
be an annihilator of class group $C\ell(L)$ of $L$. For any extension $E$ of $L$ such that $E/K$ is Galois, let $H=Gal(E/L)$ and put 
$$\theta=\sum_{i=1}^s\sum_{j=1}^r \sigma_i \tau_j.$$
Here $H=\{\tau_1, \ldots, \tau_r\}$ and we also use $\sigma_i$ for the lift of $\sigma_i$ in $Gal(E/K)$.  Then it is readily seen that $\theta$ is an annihilator of the class group $C\ell(E)$ of $E$.

We remark that for any extension $L/K$, the capitulation kernel is trivial, for any intermediate field, outside divisors of $[L:K]$ (see \cite{RS10}). Theorem \ref{MT4} gives that existence of $f \in \mathbf{R}_{L/K}$ implies that the class group of the intermediate field $K_f$ of degree $f$ is trivial outside divisors of $[L:K]$.

Lastly, we make the following hypothetical connection. Suppose there are infinitely many extensions $L/\Q$ such that \\
(i) $Gal(L/\Q)$ is cyclic, \\
(ii) the degree $[L:\Q]$ is square-free,\\
(iii) the extension $L/\Q$ is totally ramified,\\
(iv) the class number of $L$ is coprime to $[L:\Q]$,\\
(v) the set $\mathbf{R}_{L/\Q}$ is non-trivial.\\
Then there are infinitely many fields with class number one. Note that, one can easily obtain extensions satisfying conditions (i)-(iii), and even (iv).

\end{document}